\documentclass[preprint,number,sort&compress,12pt]{elsarticle}

\makeatletter
\def\@author#1{\g@addto@macro\elsauthors{\normalsize%
    \def\baselinestretch{1}%
    \upshape\authorsep#1\unskip\textsuperscript{%
      \ifx\@fnmark\@empty\else\unskip\sep\@fnmark\let\sep=,\fi
      \ifx\@corref\@empty\else\unskip\sep\@corref\let\sep=,\fi
      }%
    \def\authorsep{\unskip,\space}%
    \global\let\@fnmark\@empty
    \global\let\@corref\@empty  
    \global\let\sep\@empty}%
    \@eadauthor={#1}
}
\makeatother

\usepackage{amssymb}
\usepackage{amsfonts}
\usepackage{amsmath}
\usepackage{amsthm}
\usepackage{epsfig}
\usepackage[latin1]{inputenc}

\parskip=\medskipamount
\newtheorem{theorem}{Theorem}[section]
\newtheorem{lemma}[theorem]{Lemma}
\newtheorem{corollary}[theorem]{Corollary}

\theoremstyle{definition}
\newtheorem{definition}[theorem]{Definition}
\theoremstyle{remark}

\newcommand{\C}{\mathbb{C}}

\newcommand{\N}{\mathbb{N}}
\newcommand{\R}{\mathbb{R}}
\newcommand{\Z}{\mathbb{Z}}
\newcommand{\zak}{\mathrm{Z}}

\newcommand{\gabor}{\mathcal G(g,\alpha,\beta)}

\newcommand{\esssup}{\mathop{{\rm ess}\,{\rm sup}}\limits}
\newcommand{\essinf}{\mathop{{\rm ess}\,{\rm inf}}\limits}
\newcommand{\abs}[1]{\left|#1\right|}
\newcommand{\norm}[1]{\left\|#1\right\|}

\newcommand{\id}{\mathrm{id}}
\renewcommand{\Re}{\mathrm{Re}}
\renewcommand{\Im}{\mathrm{Im}}
\newcommand{\ds}{\displaystyle}

\newcommand{\AST}[2]{\underset{#1}{\overset{#2}{\mbox{\raisebox{-0.105em}{\LARGE$*$\rule[-0.11em]{0pt}{0.76em}}}}}}

\begin{document}

\begin{frontmatter}
\title
{Zeros of the Zak transform of totally positive functions\tnoteref{t1}}

\tnotetext[t1]{AMS subject classification 2000: 42C15,41A15,42C40,30E10}

\author{Tobias Kloos}
\ead{tobias.kloos@math.tu-dortmund.de}
\address{Faculty of Mathematics, TU Dortmund, D-44227 Dortmund}

\begin{abstract}
We study the Zak transform of totally positive (TP) functions.
We use the convergence of the Zak transform of TP functions of finite type
to prove that the Zak transforms of all TP functions without Gaussian factor 
in the Fourier transform have only one 
zero in their
fundamental domain of quasi-periodicity.
Our proof is based on complex analysis, especially the
Theorem of Hurwitz and some real analytic arguments, where
we use the connection of TP functions of finite type and
exponential B-splines.
\end{abstract}


\end{frontmatter}


\section*{Introduction}

The Gabor transform provides an important tool for the ana\-lysis of 
a given signal $f:\R\to\C$ in time and frequency. 
A window function $g\in L^2(\R)$ has
 time-frequency shifts
$$
    M_\xi T_y g(x)= e^{2\pi i \xi x} g(x-y),\qquad \xi,y\in\R.
$$
The Gabor transform of a square-integrable signal $f$ is defined as 
$$
   \mathcal{S}_g f(k\alpha,l\beta)
	= \langle f,M_{l\beta} T_{k\alpha} g\rangle,\qquad k,l\in\Z,
$$
where the parameters 
$(k\alpha,l\beta)\in\alpha\Z\times \beta\Z$ of the time-frequency shifts of
$g$ form a lattice in $\R^2$, with 
lattice parameters $\alpha,\beta>0$. The family
$$\mathcal G(g,\alpha,\beta) := \{M_{l\beta} T_{k\alpha}\,g \mid k,l\in\Z \}$$
is called a Gabor family. 
If there
exist constants $A,B>0$, which depend on $g,\alpha,\beta$, such that for 
every $f\in L^2(\R)$ we have
\begin{equation}\label{eq:framecond}
   A\|f\|^2 \le \sum_{k,l\in\Z} |\langle f, M_{l\beta} T_{k\alpha} g\rangle|^2
	\le B\|f\|^2,
\end{equation}
the family is called a Gabor frame. In order to describe Gabor families of a window function $g$, 
a very helpful tool is the Zak transform
\begin{equation*}
\mathrm Z_{\alpha}g(x,\omega) := \sum_{k\in\Z} g(x-k\alpha) e^{2\pi ik\alpha\omega},
\qquad (x,\omega)\in\R^2.
\end{equation*}
In Approximation Theory, the Zak transform  was 
used by Schoenberg \cite{Schoenberg:1973} 
in connection with cardinal spline interpolation.
For a polynomial 
B-spline $B_m$ of degree $m-1$, Schoenberg called $Z_1B$ 
the {\em exponential Euler spline}. 
The Zak transform $Z_\alpha g$ has the properties
\begin{equation}\label{eq:Zak1}
\mathrm Z_{\alpha}g(x,\omega+\tfrac{1}{\alpha}) = \mathrm Z_{\alpha}g(x,\omega),
\qquad
\mathrm Z_{\alpha}g(x+\alpha,\omega) = 
e^{2\pi i\alpha \omega}\, \mathrm Z_{\alpha}g(x,\omega).
\end{equation}
Therefore, its values in the lattice cell $[0,\alpha)\times [0,\tfrac{1}{\alpha})$
define $Z_\alpha g$ completely.  A well-known result for Gabor families $\gabor$, with
$\alpha=1$, $\beta=1/N$, and $N\in\N$, states that
the values  
\begin{equation}\label{eq:Aopt}
   A_{\rm opt}=\essinf_{x,\omega\in[0,1)} \sum_{j=0}^{N-1} 
	|Z_1 g(x,\omega+\tfrac{j}{N})|^2,\qquad
	   B_{\rm opt}=\esssup_{x,\omega\in[0,1)} \sum_{j=0}^{N-1} 
	|Z_1 g(x,\omega+\tfrac{j}{N})|^2,
\end{equation}
are the optimal frame-bounds $A,B$ in the inequality \eqref{eq:framecond}, whenever they are
positive and finite, see \cite[page 981]{Dau:1990}, \cite{Janssen:2003.2}. For rational values of 
$\alpha\beta$, a connection of the Zak transform $Z_\alpha g$ with 
the frame-bounds of the Gabor family $\gabor$ 
was given by Zibulsky and Zeevi \cite{ZibZee:1997}.
Therefore, the presence and the
location of zeros
of $Z_\alpha  g$ is relevant for the existence of lower frame-bounds in 
\eqref{eq:framecond}. Moreover, the celebrated Balian-Low theorem
\cite{Dau:1990} 
states that a Gabor family at the critical density $\alpha\beta=1$ 
cannot be a frame, if the window function $g$ or its Fourier transform 
$$
    \hat g(\omega)= \int_\R g(x) e^{-2\pi i x \omega}\,dx
$$
is continuous and in the \textit{Wiener space}, 
\begin{equation*}
W(\R):= \{ g\in L^{\infty}(\R) \mid \norm{g}_W=\sum_{n\in\Z} \esssup_{x\in[0,1]} \abs{g(x+n)} < \infty \}.
\end{equation*}
 The proof in \cite{Janssen:1982} uses the
topological argument, that every continuous 
function with the property \eqref{eq:Zak1} must have a zero in every lattice cell and, hence, $A_{\rm opt}$ in \eqref{eq:Aopt} with $N=1$ is zero.
In \cite{BanGroeStoe:2013} the connection of 
the Zak transform and discretized Gabor frames is summarized and it is pointed out, that
the knowledge of the location of the zeros is sufficient for providing that a
periodized and sampled window function generates a discrete Gabor frame.

Motivated by the results in \cite{KloStoe:2014}, we show that every Zak transform of  
a totally positive function without a Gaussian factor in
its Fourier transform has exactly one zero in its fundamental domain, which appears at $\omega = \tfrac{1}{2}$.
For this, in the first two sections we give a short introduction on totally positive functions, exponential B-splines and how they are
related to each other. In section three we prove the stated conjecture by using several convergence properties of totally positive
functions of finite type.


\section{Totally positive functions}

In this section, we want to introduce totally positive functions in Schoenberg's terminology and give some remarkable properties. For more detailed information on total positivity in terms of functions and matrices see \cite{Karl:1968}.

\begin{definition}[Totally positive (TP) function, \cite{Schoenberg:1947}]
A measurable, non-constant function $g:\R\rightarrow\R$ is called totally positive (TP), if for every $N\in\N$ and two sets of real numbers
$$x_1<x_2<\ldots<x_N\ ,\ \ \ \ \ y_1<y_2<\ldots<y_N,$$
the corresponding matrix $A=\bigl(g(x_j-y_k)\bigr)_{j,k=1}^N$ has a non-negative determinant.
\end{definition}

The most popular examples of TP functions are the exponential functions $e^{ax}$, $a\in\R\setminus\{0\}$, which are not integrable, the one- and two-sided exponentials $e^{-bx}\chi_{[0,\infty)}(x)$, $e^{-b\abs{x}}$, $b>0$, and the Gaussian $e^{-x^2}$, which are in $L^1(\R)$. Schoenberg gave a characterization of TP functions by their two-sided Laplace transforms and specified the subclass of integrable TP functions as follows.

\begin{theorem}[\cite{Schoenberg:1947}, \cite{Schoenberg:1951}]
A function $g:\R\rightarrow\R$, which is not an exponential $g(x) = Ce^{ax}$ with $C,a\in\R$, is a TP function, if and only if 
its two-sided Laplace transform exists in a strip $S=\{s\in\C\mid\alpha < \mathrm{Re}\,s < \beta\}$ with $-\infty\leq\alpha<\beta\leq\infty$ 
and is given by
$$(\mathcal Lg)(s) = \int_{-\infty}^{\infty} g(t) e^{-st}\,dt = Cs^{-n}e^{\gamma s^2-\delta s} 
\prod_{\nu=1}^{\infty} \frac{e^{a_{\nu}^{-1}s}}{1+a_{\nu}^{-1}s},$$
where $n\in\N_0$ and $C,\gamma,\delta,a_{\nu}$ are real parameters with
$$C>0,\ \ \gamma\geq0,\ \ a_{\nu}\neq 0,\ \ 0<\gamma+\sum_{\nu=1}^{\infty}\left(\tfrac{1}{a_{\nu}}\right)^2<\infty.$$
Moreover, $g$ is integrable and TP, if and only if its Fourier transform is given by
$$\hat{g}(\omega) =  \int_{-\infty}^{\infty} g(t) e^{-2\pi it\omega}\,dt 
=C e^{-\gamma\omega^2} e^{-2\pi i\delta\omega} 
\prod_{\nu=1}^{\infty}\frac{e^{2\pi ia_{\nu}^{-1}\omega}}{1+2\pi ia_{\nu}^{-1}\omega},$$
with the same conditions on $C,\gamma,\delta,a_{\nu}$ as above.
\end{theorem}

Unless otherwise specified, we will consider integrable TP functions with $\gamma =0$ and distinguish the infinite and finite type. Thus, by disregarding scaling and shifting, we focus on functions $g,g_n\in L^2(\R)$, given by their Fourier transforms
\begin{align}\label{TPfun}
\hat{g}(\omega) = \prod_{\nu=1}^{\infty} \frac{e^{2\pi i\frac{\omega}{a_{\nu}}}}{1+2\pi i\frac{\omega}{a_{\nu}}},\quad\hat{g}_n(\omega) = \prod_{\nu=1}^{n} \frac{e^{2\pi i\frac{\omega}{a_{\nu}}}}{1+2\pi i\frac{\omega}{a_{\nu}}},\ n\in\N,
\end{align}
where $(a_{\nu})_{\nu\in\N}\subset\R\setminus\{0\}$ and $\sum_{\nu=1}^\infty a_{\nu}^{-2}<\infty$. By applying the inverse Fourier transform to (\ref{TPfun}), we get the following expressions in the time domain
\begin{align}\label{TPfun2}
g(x) = \AST{\nu=1}{\infty} \abs{a_{\nu}}\,e^{-a_{\nu}(x+\frac{1}{a_{\nu}})}\ \chi_{[0,\infty)}\bigl(\mathrm{sign}(a_{\nu})(x+\tfrac{1}{a_{\nu}})\bigr),\\\notag
g_n(x) = \AST{\nu=1}{n} \abs{a_{\nu}}\,e^{-a_{\nu}(x+\frac{1}{a_{\nu}})}\ \chi_{[0,\infty)}\bigl(\mathrm{sign}(a_{\nu})(x+\tfrac{1}{a_{\nu}})\bigr).
\end{align}
Following Schoenberg, we define the reciprocals of their Laplace transforms as $\Psi$ or $\Psi_n$,
\begin{align}\label{TypII}
\Psi(s) = \prod_{\nu=1}^\infty (1+a_{\nu}^{-1}s)e^{-a_{\nu}^{-1}s}\quad,
\Psi_n(s) = \prod_{\nu=1}^n (1+a_{\nu}^{-1}s)e^{-a_{\nu}^{-1}s}\quad,\sum_{\nu=1}^\infty a_{\nu}^{-2} < \infty.
\end{align}
Obviously, these are entire functions with zeros only on the real line. 
By using their Laplace transforms, Hirschman and Widder in \cite[Theorem 4a]{HirschWid:1949} and Schoenberg in \cite{Schoenberg:1951} show independently from each other, that
$$\lim_{n\rightarrow\infty}g_n(x) = g(x),$$
uniformly on $\R$. Since we want to extend the result of the supposed convergence of TP functions to convergence of their Zak transforms in Section~\ref{Sec:Zak}, we need to adapt the proof by Hirschman and Widder. We start with a detailed inspection of the Laplace transforms. Since $\abs{1-(1-z)e^z} \leq \abs{z}^2$, for all $z\in\C$ with $\abs{z}\leq 1$, it is easy to see, that
$$\sum_{\nu=1}^{\infty}\abs{1-(1+a_{\nu}^{-1}s)e^{-a_{\nu}^{-1}s}}, \quad s\in\C,$$
converges locally uniformly in $\C$ and so does
$$\prod_{\nu=1}^{\infty} (1+a_{\nu}^{-1}s)e^{-a_{\nu}^{-1}s}.$$
Moreover the following was proved.

\begin{lemma}[{{\cite[Theorem 2b]{HirschWid:1949}}}]\label{HirschWid}
For any number $\tau_0\in\R\setminus\{0\}$ and any integer $p$, there exists a constant $M_p>0$ with
$$\abs{\frac{1}{\Psi(\omega+i\tau)}} \leq M_p\abs{\tau}^{-p},\quad \rm{for\ all\ } \abs{\tau}\geq \abs{\tau_0},$$
and for all $n\geq p$
$$\abs{\frac{1}{\Psi_n(\omega+i\tau)}} \leq M_p\abs{\tau}^{-p},\quad \rm{for\ all\ } \abs{\tau}\geq \abs{\tau_0},$$
uniformly for $\omega$ in any compact interval of existence.
\end{lemma}

With this, we can give some more information about the rate of convergence of sequences of TP functions.

\begin{theorem}\label{convrate}
Let $g$ be a TP function of infinite type, $g_n$ the TP function of type $n$ as in (\ref{TPfun}) and let $a_0:=\ds\min_{\nu\in\N}\abs{a_{\nu}}$. Then for $0\leq\sigma<a_0$
$$\lim_{n\rightarrow\infty}\abs{g(x)-g_n(x)}e^{\sigma\abs{x}}= 0 ,$$
uniformly for $x\in\R$.
\end{theorem}

\begin{proof}
Since the Laplace transforms of the given functions exist and are holomorphic in the strip with $-a_0< \Re\, s < a_0$, the inverse transforms provide for $0\leq\sigma<a_0$ and positive $x$
\begin{align*}
\abs{g(x)-g_n(x)} &= \abs{\frac{1}{2\pi i}\, \int_{-\sigma-i\infty}^{-\sigma+i\infty} \left(\frac{1}{\Psi(s)}-\frac{1}{\Psi_{n}(s)}\right)e^{xs}ds }\\
 &\leq \frac{e^{-\sigma x}}{2\pi} \int_\R\abs{\frac{1}{\Psi(i\tau-\sigma)}-\frac{1}{\Psi_{n}(i\tau-\sigma)}}d\tau.
\end{align*}
To show the convergence of the last integral, we split it into two integrals over the interval $I=[-R,R]$ and $\R\setminus I$, for $R>0$. For any given $\varepsilon>0$ there exists an $n_0\in\N$, such that for every $n\geq n_0$ the first part yields
\begin{align*}
 \int_{-R}^R \underbrace{\abs{\frac{1}{\Psi(i\tau-\sigma)}-\frac{1}{\Psi_{n}(i\tau-\sigma)}}}_{\longrightarrow 0\ (n\rightarrow\infty) \text{ uniformly in } [-R,R]}d\tau \leq 2\varepsilon R.
\end{align*}
With Lemma \ref{HirschWid}, for any integer $p$ and every $n>p$ the remaining integral can be estimated by
\begin{align*}
 \int_{\R\setminus I} \abs{\frac{1}{\Psi(i\tau-\sigma)}-\frac{1}{\Psi_{n}(i\tau-\sigma)}} d\tau &\leq 4M_p \int_R^{\infty} \tau^{-p}d\tau\\
 &=\frac{4M_p}{p-1}\,R^{-p+1},
\end{align*}
where $M_p$ is a constant independent of $n$. Now, choosing $p\geq 2$ and $R$ arbitrarily large completes the proof. The case of negative $x$ is given analogously by integrating over the parallel line to the imaginary axis, which includes $\sigma$.
\end{proof}
Note that the stated proof shows the same convergence properties for all the derivatives $g_n^{(r)}\rightarrow g^{(r)}$, by including a factor $s^r$ in the integrand and choosing $n\in\N$ sufficiently large.

Next, we take a closer look at the finite type, which is very interesting for implementing and computational usage of TP functions and also may reveal a connection to exponential B-splines. Since there are only finitely many exponential terms in (\ref{TPfun}), which realizes a shift by $\sum_{\nu=1}^n a_{\nu}$ in the time domain, we disregard them for the moment. Expression (\ref{TPfun2}) then simplifies to
$$g_n(x) = \AST{\nu=1}{n} \abs{a_{\nu}}\,e^{-a_{\nu}x}\ \chi_{[0,\infty)}\bigl(\mathrm{sign}(a_{\nu})\,x\bigr).$$
If $b_1,\ldots,b_r$, $b_i\neq b_j$, are distinct and $\mu_1,\ldots,\mu_r$ are the associated multiplicities, such that
$$\{a_1,\ldots,a_n\} = \{\underbrace{b_1,\ldots,b_1}_{\mu_1},\ldots,\underbrace{b_r,\ldots,b_r}_{\mu_r}\},$$
this notation implies, that
\begin{align}
g_n(x) = \begin{cases}\label{darstellungexp}
\ds\sum_{b_i>0} p_{b_i}(x)e^{-b_i x} & ,x\geq 0,\\
\ds\sum_{b_i<0} p_{b_i}(x)e^{-b_i x} & ,x\leq 0,
\end{cases}
\end{align}
where $p_{b_i}$ are polynomials of degree $\mu_i-1$. To be more precise, by using divided differences, St\"ockler \cite{Stoe:2012} gives the following closed form for these functions
\begin{align*}
g_n(x) &= (-1)^{n-1}\,\mathrm{sign}(x)\,\left(\prod_{\nu=1}^n a_{\nu}\right)\, [ a_1,\ldots,a_n\mid e^{-x\cdot}\chi_{[0,\infty)}(x\cdot)]\quad &,x\neq0,\\
g_n(x) &= (-1)^{n-1}\,\left(\prod_{\nu=1}^n a_{\nu}\right)\, [ a_1,\ldots,a_n\mid \chi_{[0,\infty)}(\cdot)]\quad &,x=0.
\end{align*}
With this form and the well-known identity
$$ [ a_1,\ldots,a_n\mid f] = \sum_{i=1}^r\sum_{j=1}^{\mu_r} c_{i,j}\, f^{(j-1)}(a_i)\quad ,\ c_{i,j}\in\R,\ c_{i,\mu_i} \neq 0,$$
of divided differences, see e.g. \cite{Schum:1981}, it is possible to compute the polynomials in (\ref{darstellungexp}). In the case of distinct weights, the formula reduces to
$$ [ a_1,\ldots,a_n\mid f] = \sum_{i=1}^n \prod_{j=1\atop j\neq i}^n (a_i-a_j)^{-1}\, f(a_i),$$
which provides an explicit form for $g_n$, under these assumptions. In \cite{BanGroeStoe:2013} Bannert, Gr\"ochenig and St\"ockler also computed this explicit form of the coefficients, by using the partial fraction decomposition of $\hat{g}_n$, instead  of divided differences.
All in all, these representations provide a good way for implementation and computing these functions for usage in time-frequency analysis. Moreover it reveals the close resemblance to exponential B-splines, which will be defined in the next section. 


\section{Exponential B-splines}

We want to give a short summary of exponential B-splines and list some important features of them. For a more detailed introduction, see
e.g. \cite{Schum:1981}, \cite{Ron:1987} and \cite{KloStoe:2014}.

For a set of weights
$$
   \Lambda=(\underbrace{\eta_1,\ldots,\eta_1}_{\mu_1},\ldots,
	\underbrace{\eta_r,\ldots,\eta_r}_{\mu_r}),
$$
with  pairwise distinct $\eta_1,\ldots,\eta_r\in\R$, and each $\eta_j$ 
repeated with multiplicity $\mu_j\in\N$, the space $\mathcal U_m$, given by
$$\mathcal U_m = 
\mathrm{span}\left( e^{\eta_1 x},xe^{\eta_1 x},\ldots,x^{\mu_1-1}e^{\eta_1 x},\ldots,e^{\eta_r x},\ldots,x^{\mu_r-1}e^{\eta_r x} \right),$$
forms an extended complete Tschebycheff (ECT) space. This means, that there exists a basis $\{u_1,\ldots,u_m\}\in\mathcal U_m$, such that
$$
\det\left(M\begin{pmatrix}
    u_1,\ldots,u_\ell\\
    t_1,\ldots,t_\ell
   \end{pmatrix}\right)>0
$$
for all 
$1\le \ell\le m$ and $t_1\leq\ldots\leq t_\ell \in \R$. The matrix 
$$M\begin{pmatrix}
    u_1,\ldots,u_\ell\\
    t_1,\ldots,t_\ell
   \end{pmatrix}$$ 
is the collocation matrix of Hermite interpolation, if some nodes 
	$t_j$ coincide. In our case this basis is given by the exponentials.

For ECT spaces in general it is an important result in spline interpolation theory, that there exist functions, which are compactly supported, piecewise in these spaces and sufficiently smooth. These properties are fulfilled by B-splines. In our special case, the associated splines are defined as follows.

\begin{definition}
Let $\Lambda=(\lambda_1,\ldots,\lambda_m)\in\R^m$, $\lambda_0=0$, and let exponential weight functions
$$w_j := e^{(\lambda_j-\lambda_{j-1})x},\ \ j=1,\ldots,m,$$
be given. Then, with proper normalization, the exponential B-spline (EB-spline) $B_\Lambda$, with knots $0,1,\ldots,m$, is given by the convolution of the functions $e^{\lambda_j(\cdot)}\chi_{[0,1)}$,
\begin{align}\label{expBspline}
B_\Lambda =  e^{\lambda_1(\cdot)}\chi_{[0,1)} \, \ast e^{\lambda_2(\cdot)}\chi_{[0,1)} \,
 \ast \ldots \ast \, e^{\lambda_m(\cdot)}\chi_{[0,1)},
\end{align}
so its Fourier transform is given by
\begin{align*}
\hat{B_\Lambda}(\omega) = \prod_{j=1}^m \frac{e^{\lambda_{j}-2\pi i\omega}-1}{\lambda_{j}-2\pi i\omega}.
\end{align*}
\end{definition}

By the definition, it is easy to see, that $B_{\Lambda} \in C^{m-2}(\R)$ and ${\rm supp}\, B_{\Lambda} = [0,m]$. 
Moreover, $B_{\Lambda}\mid_{(j,j+1)} \in \mathcal U_m$, $0\le j\le m-1$, so they are piecewise exponential sums
\begin{equation}\label{eq:Bpiecewise}
  B_\Lambda(x+k-1)=\sum_{j=1}^r p_j^{(k)}(x) e^{\eta_j x},\qquad x\in[0,1),~1\leq k\leq m,
\end{equation}
with real polynomials $p_j^{(k)}$ of degree $\mu_j-1$ (cf. (\ref{darstellungexp})). For our purpose, we define the differential operators
\begin{equation*}
{L}_j f = \frac{d}{dx}\left( \frac{f}{w_j} \right),\ 
\ \mathfrak L_j = {L}_j\cdots {L}_1,\quad j=1,\ldots,m,
\end{equation*}
as in \cite[page 365]{Schum:1981} and take note, that $\mathcal U_m$ is the kernel of $\mathfrak L_m$. Furthermore, these operators can also be written as
$$
   \mathfrak L_j =e^{-\lambda_{j}x}\prod_{k=1}^j\left(\frac{d}{dx}-\lambda_{k}\,\id\right).
$$
Up to normalization, $\mathfrak L_1B_{\Lambda}$ denotes the first reduced EB-spline and is given by
\begin{align*}
B_{\{\lambda_2-\lambda_1,\ldots,\lambda_m-\lambda_1\}} =  e^{(\lambda_2-\lambda_1)(\cdot)}\chi_{[0,1)} \,
 \ast \ldots \ast \, e^{(\lambda_m-\lambda_1)(\cdot)}\chi_{[0,1)}.
\end{align*}


\section{The Zak transform of totally positive functions} \label{Sec:Zak}

The Zak transform is an important tool in Gabor analysis which is commonly applied in spline theory. For a parameter $\alpha >0$ and a function $f:\R\rightarrow\C$ it is defined by
$$\mathrm{Z}_{\alpha}f(x,\omega) := \sum_{k\in\Z}f(x+\alpha k) e^{-2\pi ik\alpha \omega},$$
whenever this series exists. In the following we need the properties below of this transform. More facts can be found in \cite{Groech:2001}.

\begin{lemma}\label{Zakprop}
Let $f$ be an element of the Wiener space $ W(\R)$.
\begin{itemize}
 \item[a) ] $\mathrm Z_\alpha f(x,\omega)$ is bounded in $\R^2$, 
and if $f$ is continuous, then $\mathrm Z_\alpha f$ is continuous.
  \item[b) ] For every $n\in\Z$, we have the identities for periodicity
$$\mathrm Z_{\alpha}f(x,\omega+\tfrac{n}{\alpha}) = \mathrm Z_{\alpha}f(x,\omega)$$
and quasi-periodicity
$$\mathrm Z_{\alpha}f(x+n\alpha,\omega) = 
e^{2\pi in\alpha \omega}\, \mathrm Z_{\alpha}f(x,\omega).$$
 \item[c) ] If $\hat{f}\in W(\R)$ as well, then
$$\alpha\cdot \mathrm Z_{\alpha}f(x,\omega) = e^{2\pi ix\omega}\,\mathrm Z_{1/\alpha}\hat{f}(\omega,-x).$$
 \item[d) ] Let $f_{\alpha}=f(\alpha\cdot)$ be the scaled function of $f$. Then
$$\mathrm Z_{\alpha}f(x,\omega) = \mathrm Z_1f_{\alpha}(\tfrac{x}{\alpha},\alpha\omega).$$
\end{itemize}
\end{lemma}

Because of the last property and since the considered function spaces are scaling invariant, we restrict ourselves to the case $\alpha = 1$ and therefore write $\mathrm Zf$, instead of $\mathrm Z_1f$. 
We also use the following connection of the Zak transform of TP functions of finite type and EB-splines, which was found recently in \cite{KloStoe:2014}.

\begin{theorem}[{\cite[Theorem 3.4]{KloStoe:2014}}]\label{b-splinezak}
Let $g\in L^1(\R)$ be a TP function of finite type, defined by its Fourier transform
$$\hat{g}(\omega) = \prod_{\nu=1}^m  (1+2\pi i \tfrac{\omega}{a_{\nu}})^{-1},$$
where $a_1,\ldots,a_m\in\R\setminus\{0\}$. 
With $\lambda_{\nu}:=-a_{\nu}$ and $B_\Lambda$ defined as in (\ref{expBspline}), we have
$$\,\mathrm Zg(x,\omega) = \prod_{\nu=1}^m \, \frac{a_{\nu}}{1-e^{-(a_{\nu}+2\pi i\omega)}}\, 
\mathrm ZB_\Lambda(x,\omega),\ \ (x,\omega)\in[0,1)\times[0,1).$$
\end{theorem}

Note that this factorization of the Zak transform of TP functions of finite type holds on $\R^2$ by quasi-periodicity of the transform. 

Next, we complexify the argument $\omega$ of the Zak transform and show the convergence of the transforms of TP functions of finite type to the transforms of the related TP functions of infinite type.

\begin{definition}
The complexified Zak transform of a function $f\in L^2(\R)$ is given by 
\begin{align*}
\mathrm{Z}f(x,s) = \sum_{k\in\Z} f(x+k)\,e^{-2\pi i ks} = \sum_{k\in\Z} f(x+k)\,e^{-2\pi i k\omega}e^{2\pi k\tau},
\end{align*}
where $x,\omega,\tau\in\R$, $s=\omega + i\tau\in\C$ are chosen, such that the series converges.
\end{definition}

Schoenberg proved in \cite{Schoenberg:1951}, that integrable TP functions $\tilde{g}$ decay exponentially,
\begin{align*}
\lim_{x\rightarrow\pm\infty}e^{xs}\tilde{g}(x) = 0,\qquad -a_0<s< a_0.
\end{align*}
Hence, their complexified Zak transforms exist for all $\abs{\tau} < \frac{a_0}{2\pi}$.

\begin{theorem}\label{Zakconv}
Let $g$ be a TP function of infinite type and $g_n$ the TP function of type $n$ as in (\ref{TPfun}). Then
for a fixed $x_0\in[0,1)$ the transforms $\mathrm Zg_n(x_0,\cdot)$ and $\mathrm Zg(x_0,\cdot)$ are holomorphic in the strip $S_\xi=\{s\in\C\mid\abs{\Im(s)} \leq \xi\}$, 
whenever $0\leq\xi < \frac{a_0}{2\pi}$, and 
$$\lim_{n\rightarrow\infty} \abs{\mathrm Zg(x,s)-\mathrm Zg_n(x,s)} = 0,$$
uniformly for all $s=\omega+i\tau$ in the strip $S_\xi$, $0\leq \xi < \frac{a_0}{2\pi}$, and all $x\in[0,1)$.
\end{theorem}

\begin{proof}
The holomorphism of $\mathrm Zg_n(x_0,\cdot)$ of TP functions of finite type easily follows from Theorem~\ref{b-splinezak}, which provides an expression as a finite sum of exponentials, multiplied with a function with singularities outside of $S_{\xi}$. The holomorphism in the case of TP functions of infinite type is given by the uniform convergence, which we prove next.

Let $x\in[0,1)$ and $c\in\R$ with $2\pi\xi<2\pi c<a_0$. By Theorem \ref{convrate}, for a given $\varepsilon>0$, there exists $n_0\in\N$, such that for every $n\geq n_0$
\begin{align*}
\abs{\mathrm Zg(x,s)-\mathrm Zg_n(x,s)} &\leq \sum_{k\in\Z}\abs{g(x+k)-g_n(x+k)}\,e^{2\pi\abs{k\tau}}\\
&\leq \varepsilon\ \sum_{k\in\Z}e^{-2\pi c\abs{x+k}}\, e^{2\pi\abs{k\tau}}\\
&\leq \varepsilon\, \left( \sum_{k=1}^{\infty} e^{2\pi k(\abs{\tau}-c)} + e^{-2\pi cx} + e^{2\pi c}\,\sum_{k=-\infty}^{-1} e^{2\pi k(c-\abs{\tau})} \right)\\
&\leq \varepsilon\, \left(1+e^{2\pi c}\right) \sum_{k=0}^{\infty}e^{2\pi k(\abs{\tau}-c)} \leq \varepsilon\, \left(1+e^{2\pi c}\right)\, \frac{e^{2\pi(c-\abs{\tau})}}{2\pi(c-\abs{\tau})}.
\end{align*}
This proves the second part of the Theorem and implies, that the Zak transform of a TP function $\mathrm Zg(x_0,\cdot)$ of infinite type is the uniform limit of holomorphic functions, in $S_\xi$, which completes the first part of the proof.
\end{proof}

Next, we will use this convergence property to show, by arguments of complex analysis, that the transforms have exactly one zero in their fundamental domain. For this we will use the Theorem of Hurwitz.

\begin{lemma}[Hurwitz]\label{Hurwitz}
Let $D\subset\C$ be a domain and $(f_n)$ be a sequence of holomorphic functions in $D$, which converges locally uniformly to a function $f$. If every $f_n$ has at most $k$ zeros in $D$, then $f$ has at most $k$ zeros or $f(z)=0$, for all $z\in D$.
\end{lemma}

The following result in \cite{KloStoe:2014} shows, that the Zak transform of TP functions of finite type $m\geq 2$ has exactly one zero in the fundamental domain. 

\begin{theorem}[{\cite[Corollary 3.5]{KloStoe:2014}}]\label{zakzeroesg}
Let  $g_n$ be a totally positive function of finite type $n\ge 2$. 
Then there exists $\tilde{x}\in[0,1)$, such that 
$\mathrm Zg_n(\tilde{x},\tfrac{1}{2}) = 0$, and $\mathrm Zg_n(x,\omega) \neq 0$ for all 
$(x,\omega) \in [0,1)^2\setminus\{(\tilde{x},\tfrac{1}{2})\}$.
\end{theorem}
Next, we extend this result and show, that the complexified Zak transform has no zero in $[0,1)\times \{s=\omega+i\tau\mid \abs{\omega}<\frac{1}{2},\abs{\tau}<\frac{a_0}{2\pi}\}$. Then Lemma~\ref{Hurwitz} implies that this property holds for TP functions of infinite type. The main idea of the proof is counting sign changes and using the quasi-periodicity to construct a contradiction of having a zero in the given domain. Following \cite{deBoor:1976}, we say that a function $f:[a,b]\to\R$ 
has at least $p$ strong sign changes, if there exists a nondecreasing sequence
$(\tau_j)_{0\le j\le p}$ in $[a,b]$ with $f(\tau_0)\ne 0$ and, in case $p\ge 1$, 
$f(\tau_{j-1})f(\tau_j)<0$ for all $j=1,\ldots,p$. 
The supremum of the number of strong sign changes of $f$  is denoted by $S^-(f)$.
Similarly, we define the total number of sign changes $S^-(c)$ 
of a sequence of real numbers $c=(c_k)_{0\le k\le N}$.

\begin{theorem}\label{complZakzero}
Let $\tilde{g}$ be a continuous TP function, defined by (\ref{TPfun}), including the infinite type. Then for the complexified Zak transform, it holds that
$$\mathrm Z\tilde{g}(x,s) \neq 0 \qquad ,\text{ for }x\in[0,1)\text{ and } s\in(-\tfrac{1}{2},\tfrac{1}{2})\times i(-\tfrac{a_0}{2\pi},\tfrac{a_0}{2\pi}).$$
\end{theorem}

\begin{proof}
With Theorem~\ref{b-splinezak}, the finite case $\tilde{g}=g_n$, $n\in\N$, results by proving the same property for the Zak transform of the associated EB-spline $B_\Lambda$. We will show this by following the proof in \cite{KloStoe:2014}. Again we use the notation 
$$\Lambda=(\eta_1,\ldots,\eta_1,\ldots,\eta_r,\ldots,\eta_r)$$ 
with multiplicities $\mu_j$ of  the pairwise distinct weights $\eta_j$. For a fixed $s\in(-\tfrac{1}{2},\tfrac{1}{2})\times (-\tfrac{a_0}{2\pi},\tfrac{a_0}{2\pi})$, we
consider the complex valued function $h:=\zak B_\Lambda( \cdot,s)$. 
By (\ref{eq:Bpiecewise}), we obtain for $x\in[0,1)$ 
\begin{align}\label{eq:EBspline}
  h(x)=\sum_{k=0}^{n-1} B_\Lambda(x+k)e^{-2\pi iks} &=
	\sum_{k=1}^{n} \sum_{j=1}^r p_j^{(k)}(x)e^{\eta_jx} e^{-2\pi iks} \notag\\
       &= \sum_{j=1}^r \underbrace{\sum_{k=1}^{n} p_j^{(k)}(x)e^{-2\pi iks}}_{:=q_j(x)} 
			e^{\eta_jx},
\end{align}
where $q_j$ are complex polynomials of degree $\sigma_j-1\le \mu_j-1$. 
Some of the $q_j$ are nonzero, since the shifts of EB-splines are locally linearly independent. This means that,
if $\Re(h)$ vanishes on an interval $(a,b)\in\R$, then all coefficients $\Re(e^{-2\pi iks})$ with $\mathrm{supp}B_\Lambda(\cdot+k)\cap(a,b)\neq\emptyset$
vanish as well. Since $\abs{\Re(s)}<\tfrac{1}{2}$, no consecutive coefficients of $\Re(h)$ vanish simultaneously. 
Therefore there is no non-empty interval where $\Re(h)$ is identically zero.
We let $\sigma_j=0$ if $q_j=0$
and define $\gamma=\sigma_1+\ldots+\sigma_r$. 
Without loss of generality, we can assume $\sigma_1\ge 1$, that is, the term
$q_1(x)e^{\eta_1 x}$  in $h|_{[0,1)}$ is nonzero. 
We use the identity
$$
    e^{\eta_j x} \frac{d}{dx}\left( e^{-\eta_j x}h(x)\right)=q_j'(x)e^{\eta_j x}+\sum_{k\ne j} 
		((\eta_k-\eta_j)q_k(x)+q_k'(x)))e^{\eta_k x}.
$$ 
Writing $D_j$ for the differential operator on the left hand side and \newline
$\mathcal{D}:=D_{1}^{\sigma_1-1}\prod_{j=2}^{r}D_{j}^{\sigma_j}$, we obtain
$$
  \mathcal{D} h(x)
= be^{\eta_1x},\qquad x\in (0,1),
$$
with  a nonzero constant $b\in\C$. The quasi-periodicity of $h$ leads directly to
\begin{align*}
   \mathcal{D}h(x)= be^{\eta_1(x-k)}\,e^{2\pi iks},\qquad x\in (k,k+1),
\end{align*}
for all $k\in\Z$. 

Now we assume 
that there exists $\tilde x\in[0,1)$ with
$\zak B_\Lambda(\tilde x,s)=0$. 
By quasi-periodicity of $h=\zak B_\Lambda(\cdot,s)$, the function
$ f={\rm Re}\, h$
vanishes at all points $\tilde x+k$, $k\in\Z$, and these points are isolated zeros of $f$
by local linear independence again.
This guarantees that $f$ has at least $N\in\N$ isolated zeros in $[0,N]$. 
Note that $\mathcal{D}$ is a differential operator of order $\gamma-1\le m-1$. 
Since  $f\in C^{m-2}(\R)$, with $f^{(m-2)}$ absolutely continuous,
we obtain by Rolle's
theorem that
\begin{equation}\label{eq:signchange}
  S^-(\mathcal{D}f)\ge N-\gamma+1\qquad \hbox{on}\quad [0,N].
	\end{equation}
However, 
on each interval $[k,k+1)$ with $k\in\Z$ and $s=\omega+i\tau$, $\omega,\tau\in\R$, the sign of $\mathcal{D}f$ is fixed by
$$
   {\rm sign}\,(\mathcal{D}f)(x) = 
	{\rm sign}\, {\rm Re}\left(b\,e^{2\pi ik(\omega+i\tau)}\right) = {\rm sign}\, {\rm Re}\left(b\,e^{2\pi ik\omega}\right),\qquad x\in[k,k+1).
$$
This implies
$$
   S^-(\mathcal{D}f)\le  2N|\omega|\qquad\hbox{on}\quad [0,N],$$
which is a contradiction to \eqref{eq:signchange} for $|\omega|<\tfrac{1}{2}$ and 
large $N$. This completes the proof for TP functions of finite type (and their associated EB-splines).

The Zak transform of TP functions of infinite type is not identically zero on $(-\tfrac{1}{2},\tfrac{1}{2})\times i(-\tfrac{a_0}{2\pi},\tfrac{a_0}{2\pi})$, because $g$ is positive and $\mathrm Zg(x,0)=\sum_{k\in\Z} g(x+k)>0$. Hence, Theorem~\ref{Zakconv} and Lemma~\ref{Hurwitz} implies, that it can not have any zero in this domain, since it is the uniform limit of holomorphic functions $\mathrm Zg_n(x,\cdot)$ without any zeros, which completes the proof.
\end{proof}

This Theorem especially implies, that $\mathrm Zg$ has no zero in $[0,1)\times(-\tfrac{1}{2},\tfrac{1}{2})$. Since it is well-known, that Zak transforms of continuous functions do have a zero, we are left to show, that they have exactly one zero in their fundamental domain, whenever $g$ is a TP function of infinite type without a Gaussian factor. For this, we need some preparations.

\begin{lemma}\label{features}
Let $\tilde{g}$ be a continuous TP function, defined by (\ref{TPfun}), including the infinite type. Then $\mathrm Z\tilde{g}(\cdot,\frac{1}{2})$ is real, $2$-periodic and not identically zero.
\end{lemma}
\begin{proof}
By the definition it is easy to see, that 
$$\mathrm Z\tilde{g}(x,\tfrac{1}{2}) = \sum_{k\in\Z}\tilde{g}(x+k)\,e^{-2\pi ik\tfrac{1}{2}} = \sum_{k\in\Z}\tilde{g}(x+k)(-1)^k$$
is a real valued function. Since it is quasi-periodic, 
$$\mathrm Z\tilde{g}(x+n,\tfrac{1}{2}) = e^{2\pi in\tfrac{1}{2}}\ \mathrm Z\tilde{g}(x,\tfrac{1}{2}) = (-1)^n\ \mathrm Z\tilde{g}(x,\tfrac{1}{2}),$$
obviously $\mathrm Z\tilde{g}(\cdot,\frac{1}{2})$ is a $2$-periodic function. Moreover, the Inversion Formula for Zak transforms,
\begin{align*}
\int_0^1 \mathrm Z\tilde{g}(x,\omega)e^{-2\pi ix\omega}\, dx &= \int_0^1 \left( \sum_{k\in\Z} \tilde{g}(x-k) e^{2\pi i\omega(k-x)} \right)\, dx\\
&= \sum_{k\in\Z} \int_0^1 \tilde{g}(x-k) e^{-2\pi i\omega(x-k)} \, dx\\ 
&= \hat{\tilde{g}}(\omega),
\end{align*}
implies that $\mathrm Z\tilde{g}(\cdot,\tfrac{1}{2})$ cannot be identically zero. 
\end{proof}

\begin{lemma}\label{monotonexp}
Let $\eta_1,\ldots,\eta_r\in\R$ be some pairwise distinct numbers, $q_1,\ldots,q_r$ some real nonzero polynomials of degree $\sigma_j-1$. Let $\gamma = \sigma_1 +\ldots +\sigma_r$ and consider a real valued function $h\in C^{\gamma-2}(\R)$, where $h^{(\gamma-2)}$ is absolutely continuous, $h(x+1) = -h(x)$ and
$$h(x) = \sum_{j=1}^r q_j(x)\, e^{\eta_jx},\quad x\in[0,1).$$
Then there exists $x_0\in\R$, such that $h$ is monotone on each interval $[x_0+k,x_0+k+1)$, $k\in\Z$.
\end{lemma}
\begin{proof}
Analogously to the proof of Theorem~\ref{complZakzero}, consider the differential operator $\mathcal{D}=D_{1}^{\sigma_1-1}\prod_{j=2}^{r}D_{j}^{\sigma_j}$. Due to continuity and periodicity, $h(x+2) = -h(x+1) = h(x)$, the range of $h$ is bounded and so $h$ has at least one maximum at a point $x_0$ (and one minimum) in the interval $[0,2)$ of a period. If $h$ is not monotonously decreasing on $[x_0,x_0+1)$, the $2$-periodicity implies, that there is a number $c\in\R$, such that the function $h_c := h(\cdot)+c$ has at least four zeros in $[0,2)$. If we apply $\mathcal D$ to this function we get
$$\mathcal D h_c(x) = be^{\eta_1 x} + \underbrace{\prod_{j=2}^r(-\eta_j)^{\sigma_j}\, (-\eta_1)^{\sigma_1-1}\, c}_{=:\tilde{c}}, \quad b\in\R,\ x\in(0,1),$$
and since $h(x+1) = -h(x)$, we have
$$\mathcal D h_c(x) = (-1)^k\, be^{\eta_1 x} +\tilde{c},\quad x\in (k,k+1), k\in\Z.$$
For large $N\in\N$ we get a contradiction, by counting sign changes again:
$$4N-\gamma+1\leq \mathcal S^-(\mathcal Dh_c)\leq 2N\qquad\hbox{on}\quad [0,2N].$$
This shows, that $h$ is monotone on every interval $[x_0+k,x_0+k+1)$, $k\in\Z$, where $x_0+k$ are the extrema of $h$.
\end{proof}

\begin{corollary}\label{Bmonotone}
Let $\Lambda=(\eta_1,\ldots,\eta_1,\ldots,\eta_r,\ldots,\eta_r)\in\R^m$, $\eta_j\neq\eta_i$, and $B_{\Lambda}$ be the associated EB-spline as defined in (\ref{expBspline}). Then there exist numbers $x_0,y_0\in\R$, such that 
$\mathrm ZB_{\Lambda}(\cdot,\tfrac{1}{2})$
is monotone on $[x_0+k,x_0+k+1)$ and the derivative $D_n\mathrm ZB_{\Lambda}(\cdot,\tfrac{1}{2})$, $1\leq n\leq r$, is monotone on $[y_0+k,y_0+k+1)$, for all $k\in\Z$.
\end{corollary}
\begin{proof}
With the notations as in equation (\ref{eq:EBspline}) these functions can be written as a sum of exponentials with polynomial coefficients,
\begin{eqnarray*}
\mathrm ZB_{\Lambda}(x,\tfrac{1}{2})    & = & \sum_{j=1}^r q_j(x)\,e^{\eta_jx},\\
D_n\mathrm ZB_{\Lambda}(x,\tfrac{1}{2}) & = & \left(D_n\left( \sum_{j=1}^r q_j\,e^{\eta_j\cdot} \right)\right)\!(x) = e^{\eta_nx}\frac{d}{dx}\left( e^{-\eta_nx} \sum_{j=1}^r q_j(x)\,e^{\eta_jx} \right)\\
                                        & = & q_n'(x)e^{\eta_n x}+\sum_{j\neq n} ((\eta_j-\eta_n)q_k(x)+q_k'(x)))e^{\eta_k x}.
\end{eqnarray*}
Therefore the Corollary follows directly from Lemma~\ref{features} and Lemma~\ref{monotonexp}.
\end{proof}

\begin{corollary}\label{monotonicity}
Let $g$ be a TP function of infinite type as defined in (\ref{TPfun}). Then there exists $x_0\in[0,1)$, such that $\mathrm Zg(\cdot,\tfrac{1}{2})$ is monotone on every interval  $[x_0+k,x_0+k+1)$, $k\in\Z$, of length one.
\end{corollary}
\begin{proof}
Because of continuity and periodicity the range of $f:=\mathrm Zg(\cdot,\tfrac{1}{2})$ is bounded and there is at least one maximum (and one minimum) in the interval $[0,2)$. Without loss of generality, let there be a maximum at $x_0\in[0,1)$ (and at $x_0+1$ a minimum). If $f$ is not monotonously decreasing on $[x_0,x_0+1)$, there exist $\delta>0$ and points $z_1<z_2<z_3\in[x_0+\delta,x_0+1-\delta]$, such that $f$ is monotone on $[x_0,x_0+\delta)$ and 
$$f(z_1)<f(z_2)\ ,\quad f(z_3)<f(z_2).$$
With $\varepsilon := \tfrac{1}{4}\cdot\min\{f(x_0)-f(z_1),f(z_2)-f(z_1),f(z_2)-f(z_3)\}$ the uniform convergence of the Zak transforms of TP functions implies, that there is a number $n_0\in\N$, such that for every $n\geq n_0$ the function $f_n := \mathrm Zg_n(\cdot,\tfrac{1}{2})$ fulfills
$$f_n(x_0)>f_n(z_1),\ f_n(z_2)>f_n(z_1),\ f_n(z_2)>f_n(z_3).$$
Therefore there is also no interval of length one, where $f_n$ is monotone. Because of Theorem~\ref{b-splinezak}, the same property holds for the Zak transform of the associated EB-spline $\mathrm ZB_{\Lambda}(\cdot,\tfrac{1}{2}) = C\cdot \mathrm Zg_n(\cdot,\tfrac{1}{2})$, $C\in\R$, which is a contradiction to Lemma~\ref{Bmonotone}.
\end{proof}

Now we can prove the last part of our main result.

\begin{theorem}
Let $g$ be a TP function of infinite type as defined in (\ref{TPfun}). Then there exists $\tilde{x}\in[0,1)$, such that $\mathrm Zg(\tilde{x},\tfrac{1}{2}) = 0$,
and $\mathrm Zg(x,\tfrac{1}{2}) \neq 0$ 
for all $x \in [0,1)\setminus\{\tilde{x}\}$.
\end{theorem}

\begin{proof}
It was first pointed out in \cite{Zak:1975} that if the Zak transform $\mathrm Z g$ is continuous, 
then it has a zero in $[0,1)\times[0,1)$. Since $g$ is real, we know from the results in \cite{Janssen:1988},
that there exists $\tilde{x}\in[0,1)$, such that $\mathrm Zg(\tilde{x},\tfrac{1}{2}) = 0$. Now, we assume that there exist at least two zeros in an interval of monotonicity $[x_0,x_0+1)$, where $\mathrm Zg(\cdot,\tfrac{1}{2})$ has a maximum at $x_0$ again. If there are some isolated zeros or more than one zero intervals, $\mathrm Zg(\cdot,\tfrac{1}{2})$ can not be monotone, which is a contradiction to Corollary~\ref{monotonicity}. Therefore we suppose, that there is one single zero interval $[z_1,z_2]\subset[x_0,x_0+1)$. First we assume that $a_1<0$, where
$$\hat{g}(\omega) = \prod_{\nu=1}^{\infty} \frac{e^{2\pi i\frac{\omega}{a_{\nu}}}}{1+2\pi i\frac{\omega}{a_{\nu}}},$$
and we let $f:=\mathrm Zg(\cdot,\frac{1}{2})$. Since $f$ is monotonously decreasing on $[x_0,x_0+1)$, the mean-value theorem implies, that for $0<\varepsilon<-a_1^{-1}$ there is $0<\theta<\varepsilon$ with
$$\abs{f'(z_2+\theta)} = \frac{\abs{f(z_2+\varepsilon)}}{\varepsilon} > \frac{\abs{f(z_2+\theta)}}{\varepsilon} > -a_1\abs{f(z_2+\theta)}.$$
On the other hand, $f'(x_0+1) = 0$, $f(x_0+1)< 0$ and therefore the derivative
\begin{align}\label{eq:monot}
D_1f = \left(\frac{d}{dx}+a_1\,\mathrm{id}\right)f = f' +a_1\, f
\end{align}
fulfills $D_1f(x_0)\!<\!0$, $D_1f(z_2)\!=\!0$, $D_1f(z_2+\theta)\!<\!0$ and $D_1f(x_0+1)\!>\!0$, which implies, that $D_1f$ is not monotone on any interval of length one. Analogously to the proof of Corollary~\ref{monotonicity}, the convergence property of Theorem~\ref{Zakconv} implies, that there is a number $n_0\in\N$, such that for every $n\geq n_0$ the derivative $D_1\mathrm Zg_n(\cdot,\frac{1}{2})$ also is not monotone on any interval of length one. Hence, because of Theorem~\ref{b-splinezak}, the same holds for the Zak transform of the associated EB-spline $D_1\mathrm ZB_{\Lambda}(\cdot,\tfrac{1}{2})$, which is a contradiction to Lemma~\ref{Bmonotone} again. Considering the interval $[x_0,z_1]$, instead of $[z_2,x_0+1]$, the case $a_1>0$ follows analogously.
\end{proof}

In \cite{BanGroeStoe:2013} the information about the location of the zero of the Zak transform of TP functions of finite type is used to construct discrete Gabor frames in several cases. Now, with Theorem~\ref{complZakzero}, we can extend the result to the class of all TP functions without a Gaussian factor in its Fourier transform. With the periodization operator $\mathcal P_K:L^1(\R)\rightarrow L^1(\mathbb T_K)$ and the sampling operator $S:L^2(\R)\rightarrow \ell^2(\Z)$,
\begin{align*}
\mathcal P_Kg = \sum_{k\in\Z}g(\cdot -kK),\qquad Sg = (g(k))_{k\in\Z},
\end{align*}
the following holds.

\begin{corollary}[{cf. \cite[Theorem 8]{BanGroeStoe:2013}}]
Let $\tilde{g}$ be a continuous TP function, defined by (\ref{TPfun}), including the infinite type. Assume $\alpha = M\in\N$ and let $\beta = 1/M$ and $k\in\N$ such that $K/M\in\N$.
\begin{itemize}
\item If $K/M$ is odd, then $\mathcal G(\mathcal P_K\tilde{g},\alpha,\beta)$ is a Gabor frame for $L^2(\mathbb T_k)$.
\item If $K/M$ is odd, then $\mathcal G(\mathcal P_KS\tilde{g},\alpha,\beta)$ is a Gabor frame for $\mathbb C^K$.
\end{itemize}
In addition, assume that $\tilde{g}$ is even, which means, that $\{a_{\nu}\mid a_{\nu}>0\}=\{-a_{\nu}\mid a_{\nu}<0\}$.
\begin{itemize}
\item If $M$ is odd, then $\mathcal G(S\tilde{g},\alpha,\beta)$ is a Gabor frame for $L^2(\mathbb T_k)$.
\item If $M$ is odd, then $\mathcal G(\mathcal P_KS\tilde{g},\alpha,\beta)$ is a Gabor frame for $\mathbb C^K$.
\end{itemize}
\end{corollary}

\begin{proof}
For the proof see \cite[Theorem 8]{BanGroeStoe:2013}.
\end{proof}


\section*{Acknowledgement}

\thispagestyle{empty}
We are very grateful to J. St\"ockler for some discussions about the subjects in this article. 




\end{document}